\documentclass[pdflatex,sn-mathphys-num]{sn-jnl}

\usepackage{graphicx}%
\usepackage{multirow}%
\usepackage{amsmath,amssymb,amsfonts}%
\usepackage{amsthm}%
\usepackage{mathrsfs}%
\usepackage[title]{appendix}%
\usepackage{xcolor}%
\usepackage{textcomp}%
\usepackage{manyfoot}%
\usepackage{booktabs}%
\usepackage{algorithm}%
\usepackage{algorithmicx}%
\usepackage{algpseudocode}%
\usepackage{listings}%

\newtheorem{remark}{Remark}%

\usepackage{graphics}
\usepackage{booktabs}
\usepackage{hhline}
\usepackage{amscd}
\usepackage{setspace}
\usepackage{enumerate}
\usepackage{array,makecell}
\usepackage{mathdots}
\usepackage{xcolor}
\usepackage{pdflscape} 
\usepackage{comment}
\usepackage{upquote}
\usepackage{multicol}
\usepackage{xspace}
\usepackage[normalem]{ulem}
\usepackage{tcolorbox}
\usepackage{subcaption}
\usepackage{floatflt}
\usepackage{adjustbox}
\usepackage{hyperref}
\hypersetup{
	colorlinks   = true, 
	urlcolor     = blue, 
	linkcolor    = blue, 
	citecolor   = red 
} 
\definecolor{darkgreen}{rgb}{0.0,0.7,0.0}
\definecolor{hotpink}{rgb}{0.9,0,0.5}

\theoremstyle{definition}
\newtheorem{example}{Example}

\newtheorem{prop}{Proposition}
\makeatletter
\newcommand\footnoteref[1]{\protected@xdef\@thefnmark{\ref{#1}}\@footnotemark}
\makeatother

\newcommand {\mat}  [1] {\left[\begin{array}{#1}}
\newcommand {\rix}      {\end{array}\right]}
\def\mystrut#1{\rule{0cm}{#1}}  

\DeclareMathOperator{\diag}{diag}

\newcommand{\bRmn}{{\mathbb R}^{m \times n}}
\newcommand{\be}{\begin{equation}}
\newcommand{\ee}{\end{equation}}
\newcommand{\ba}{\begin{array}}
\newcommand{\ea}{\end{array}}
\newcommand{\R}{{\mathbb R}}

\newcommand{\mc}{\multicolumn}

\newcommand{\tnum}		{t} 
\newcommand{\pt}		{k} 
\newcommand{\Vt}		{V_{\pt}}
\newcommand{\Ut}		{U_{\pt}}
\newcommand{\At}		{A_{\pt}}
\newcommand{\St}		{\Sigma_{\pt}}

\newcommand{\tqs}{\textquotesingle}
\newcommand{\Aerrtot}	{\mathcal{E}_{\scalebox{0.7}{\hspace*{1pt}$tot$}}^{\scalebox{0.7}{\hspace*{1pt}$err$}}}
\newcommand{\AVerr}	{\mathcal{E}_{\scalebox{0.7}{\hspace*{1pt}$1$}}}
\newcommand{\ATUerr}{\mathcal{E}_{\scalebox{0.7}{\hspace*{1pt}$2$}}}
\newcommand{\UVerr}	{\mathcal{UV}^{\scalebox{0.7}{\hspace*{1pt}$err$}}}
\newcommand{\Tcputot}	{\mathcal{T}_{\scalebox{0.7}{\hspace*{1pt}$tot$}}^{\scalebox{0.7}{\hspace*{1pt}$cpu$}}}
\newcommand{\Tcpupsvd}	{\mathcal{T}_{\scalebox{0.7}{\hspace*{1pt}$psvd$}}^{\scalebox{0.7}{\hspace*{1pt}$cpu$}}}

\newcommand{\psvds}{{\tt psvd}\xspace} 
\newcommand{\psvdns}{{\tt psvd}}
 
\newcommand{\psvdzns}{{\tt psvd0}}
\newcommand{\psvdRs}{{\tt psvdR}\xspace} 
\newcommand{\psvdRns}{{\tt psvdR}}
\newcommand {\basicbox} 		{\begin{tcolorbox}[width=\textwidth,colback={black!1}]}
\newcommand {\overbox}      {\end{tcolorbox}}
\newcommand {\rbasicbox} 		
{\begin{tcolorbox}[width=\textwidth,colback={red!5}]}
\newcommand {\bbasicbox} 		
{\begin{tcolorbox}[width=\textwidth,colback={blue!5}]}
\newcommand {\gbasicbox} 		
{\begin{tcolorbox}[width=\textwidth,colback={green!10}]}
\newcommand {\ybasicbox} 		
{\begin{tcolorbox}[width=\textwidth,colback={yellow!10}]}
\newcommand {\mycomment}[1]{} 
 
\usepackage[mathlines]{lineno}  

\begin{document}

\title[Article Title]{A Hybrid Algorithm for Computing a Partial Singular Value Decomposition Satisfying a Given Threshold}

\author[1]{\fnm{James} \sur{Baglama}}\email{jbaglama@uri.edu}
\author[1]{\fnm{Jonathan} \sur{Ch\'{a}vez-Casillas}}\email{jchavezc@uri.edu}

\author[1]{\fnm{Vasilije} \sur{Perovi\'{c}}}\email{perovic@uri.edu}

\affil[1]{
\orgdiv{Department of Mathematics and Applied Mathematical Sciences}, 
\orgname{University of Rhode Island}, 
\orgaddress{
    \city{Kingston}, \state{RI}, \postcode{02881}, \country{USA}}%
}

\abstract{
In this paper, we describe 
a new hybrid algorithm for computing 
all singular triplets above a
given threshold and provide 
its implementation in 
MATLAB/Octave and {\sf R}. 
The high performance of our codes 
and ease at which they can be 
used, either independently 
or within a larger numerical scheme, are 
illustrated through several 
numerical 
examples with applications 
to matrix completion 
and image compression. 
Well-documented MATLAB and {\sf R} codes 
are provided for public use. 
}

\keywords{
partial singular value decomposition, {\tt svds},
MATLAB, {\sf R}, matrix completion, 
singular value thresholding  
}

\maketitle
\section{Introduction}\label{sec:intro}
A fundamental tool 
arising in diverse areas of applications, e.g., 
matrix completion \cite{cai2010-matrix-completion,feng2018faster,kumar2022efficient,li2017svt,li2017fast}, 
imaging \cite{asnaoui2019image,erichson2019randomized,ji2016rank,narwaria2011svd}, 
principal component analysis \cite{cherapanamjeri2017thresholding,jollIffe2005principal,xu2010robust}, 
machine learning \cite{elden2019matrix}, and genomics \cite{alter2000singular,benchmarking2020rna}, 
is the computation of a 
{\em partial singular value decomposition (PSVD)}  
of a large sparse matrix 
$A \in \R^{m \times n}$,
 \begin{equation} \label{eq:psvd}
A \, \Vt \, = \,  \Ut \, \St \, + \, \AVerr , \qquad \  
A^T  \Ut \, = \,  \Vt \, \St, + \, \ATUerr , \qquad \ 
A_k \, = \, \Ut \, \St  \Vt^T , 
\end{equation}
where 
$\St = \diag ( \sigma_1, \ldots, \sigma_{\pt})  \in \mathbb{R}^{\pt \times \pt}$ 
with singular values $\sigma_1 \geq \ldots \geq\sigma_{\pt} > 0$, 
$\Ut \in \mathbb{R}^{m \times {\pt}}$, $\Vt~\in~\mathbb{R}^{n \times {\pt}}$ 
have orthonormal columns 
of associated singular vectors, 
$A_\pt$ is a rank $\pt$ 
approximation of $A$, 
and 
$\AVerr \in \mathbb{R}^{m \times {\pt}}$,
$\ATUerr \in \mathbb{R}^{n \times {\pt}}$ 
are approximation errors.  
A singular value with its associated 
left and right singular vectors 
is referred to 
as a {\em singular triplet}. 
{In this paper we investigate 
various 
computational aspects of determining 
{\em all} singular triplets corresponding to singular values of $A$ above a 
user-specified threshold parameter {\em sigma}, 
or in other words, 
determining a $\pt$--PSVD of $A$ 
such that $\sigma_k \geq sigma$ and $\sigma_{k+1} < sigma$.} 
When $A$ is of modest size, 
a $\pt$-PSVD of $A$ can be obtained 
by truncating its full SVD  
which ultimately determines  
{\em all} singular triplets of $A$ 
(see \cite{golub1965calculating}). 
However, 
we do not consider this approach 
since the full SVD is designed for dense matrices and is not scalable  to 
large sparse matrices, 
becoming computationally 
and memory intensive. Furthermore, we may assume  
$A$ and $A^T$ are 
only accessible via 
matrix--vector product routines which can be 
optimized for respective 
sparsity  structure. 

Over the last few decades, 
various numerical schemes with 
publicly available software 
have been developed 
for finding an 
{\em approximate $\pt$--PSVD}  
of a large sparse matrix $A$, 
e.g.,  PROPACK \cite{larsen1998lanczos},
{\tt irlba} \cite{baglama2005augmented,lewis2021irlba}, 
{\tt primme\underline{\hspace{0.15cm}}svds} \cite{wu2017primme_svds}, 
{\tt svdifp} \cite{liang2014computing}, 
{\tt rsvd} \cite{erichson2019randomized}, 
{\tt svds-C} \cite{feng2024svds},
RSpectra \cite{qiu2019rspectra}, 
{\tt dashSVD} \cite{feng2024}, 
Octave's {\tt svds}\cite{octave2023}, 
and commercially available MATLAB's {\tt svds}  
\cite{MATLAB:R2024a}. 
A shared feature  by all these 
methods is that they all require 
knowing the number of 
desired singular triplets, $\pt$, 
in advance, 
thus significantly limiting their use
in applications 
where $\pt$ is not readily available. 
For example, in the now classical work 
on matrix completion by Cai, Cand\`{e}s, 
and Shen \cite{cai2010-matrix-completion}, 
the authors developed a 
method  
for the problem of recovering an  
approximately low-rank matrix 
from a sampling of its entries. 
The main computational kernel in 
their proposed method 
\cite[Alg.\! 1]{cai2010-matrix-completion} 
consists of having to 
{\em repeatedly} compute larger and larger 
PSVDs corresponding to 
singular values above a certain threshold 
(see SVT-MC Algorithm (\ref{matrix-completion-alg}) in Example \ref{example2}). 
A similar 
kernel also 
appears in an approach  for   
solving large-scale {\em linear discrete ill-posed} 
problems, 
when $A$ is very ill-conditioned 
\cite{onunwor2017computation}, 
which relies on 
solving a smaller least-squares problem 
associated with a $\pt$-PSVD of $A$, 
where $\pt$ is unknown a priori and 
is determined {\em adaptively} 
so it 
satisfies  the discrepancy principle.
The most direct approach for these problems would be to compute some predetermined 
number $\tnum$ of largest singular 
triplets by using one of the  
previously listed routines 
and check whether the threshold is met. 
If the threshold is not met,  
then current limitations of 
these routines 
necessitate having to 
{\em recompute} 
$\widetilde{\tnum} > t$ 
of largest singular triplets. 
This can be unnecessarily
expensive and it completely 
dismisses the knowledge of the 
previously computed $\tnum$ singular triplets. Our proposed algorithm completely 
circumvents this problem.

In this paper we have developed 
a hybrid algorithm that efficiently and systematically computes (an unknown number of) all singular triplets above a 
user-specified threshold. 
So far, to the best of our knowledge, 
this problem was only addressed 
either by 
adapting algorithms that work 
on a larger 
symmetric eigenproblem 
associated with 
$C = 
\big[
\begin{smallmatrix}
	0 & A \\
	A^T & 0
\end{smallmatrix}
\big]$ 
\cite{li2017svt} 
or by adaptation of various randomization 
techniques (e.g., \cite{ji2016rank, feng2023fast,
 yu2018efficient}) 
 that utilize a 
different stopping criteria 
(see Remark~\ref{rem:energy}).  
Our proposed algorithm is 
based on the {\em 
explicit deflation procedure applied 
to Golub-Kahan-Lanczos Bidiagonalization 
(GKLB)-based methods}  \cite{baglama2023explicit} 
and it greatly extends the functionality 
of the  highly popular 
MATLAB routine {\tt svds} 
\cite{MATLAB:R2024a}. 
What differentiates this algorithm from the others is that  it works 
directly on matrix $A$ 
(no need to consider  
$C = 
\big[
\begin{smallmatrix}
	0 & A \\
	A^T & 0
\end{smallmatrix}
\big]$) 
and it maintains 
the overall high accuracy 
one expects from {\tt svds}. 
Furthermore, our results  can also 
be easily adapted 
to various existing {\tt svds}-like 
implementations 
in programming languages such as 
C ({\tt svds-C} and {\tt irlb.c}), 
C{}++ ({\tt CppIrlba}), 
Python ({\tt irlb.py}), 
R ({\tt irlba.R}), and  
Julia ({\tt svdl.jl})%
.\footnote{All are public codes that can be found on GitHub: \url{https://github.com}.}

{With respect to GKLB-based methods, 
the first observation relevant to this paper is that they} 
can be thought of as 
a one--sided PSVD approximations of $A$ 
where either $\AVerr$ or $\ATUerr$ 
in \eqref{eq:psvd} is 
exactly zero while the other value 
contains the approximation error 
\cite{baglama2005augmented,larsen1998lanczos}.
{Furthermore,}
the structure of 
GKLB-based methods 
have the added advantage of 
allowing for the explicit 
deflation technique 
to  be applied only to either 
the left or right singular vectors of $A$, depending on its dimensions,  and 
thus significantly reducing 
the overall computational cost. 
Due to our iterative deflation process, 
a modest non--zero error,  
that is orthogonal to computed 
basis vectors, 
is added to the side in the GKLB 
method that had theoretically zero error 
\cite{baglama2023explicit}.

Although the error growth is 
theoretically manageable, 
our experiences 
while developing our public software on a wide--range of problems
with varying singular value distributions 
showed that not to be the case 
and thus highlighting that even 
a seemingly straightforward 
explicit deflation as in \cite{baglama2023explicit} 
requires special care when implementing.
This led us to 
develop a new, robust hybrid scheme 
that combines a GKLB-based method 
\cite{baglama2023explicit}
together with a variant of the block SVD power method \cite{bentbib2015block}, 
the latter being 
solely used 
to restore 
the one--sided GKLB structure as needed. 
This helped with maintaining 
strong orthogonality 
among  the deflated and newly 
computed singular vectors, 
computing all multiple singular values, 
and avoiding the pitfall of recomputing 
mapped singular values -- 
see Section \ref{sec:alg} 
and Example \ref{example1}. 

To make our algorithm more accessible to a wider audience, 
we developed three 
standalone and well-documented computer codes which can be 
run either independently  or as a 
part of a larger routine. 
All three codes repeatedly call a GKLB-based method based on \cite{baglama2005augmented} to compute a PSVD of 
$A$ while 
applying an explicit deflation technique. 
{In other words}, they only differ in 
which {\psvdns} method is being called 
and the programming environment.  
Specifically, our codes 
{\tt svt\underline{\hspace{0.15cm}}svds.m} (MATLAB), 
{\tt svt\underline{\hspace{0.15cm}}irlba.m} 
(MATLAB/Octave), 
and 
{\tt svt\underline{\hspace{0.15cm}}irlba.R} 
({\sf R}), 
call 
MATLAB 
native
function {\tt svds}, 
a MATLAB/Octave internal implementation 
of {\tt irlba} \cite{baglama2005augmented}, 
and the {\sf R}~package  {\tt irlba.R} 
\cite{lewis2021irlba}, 
respectively. 
All codes and documentation are available on GitHub \url{https://github.com/jbaglama/svt}.
It is important to note here 
that while both  codes 
{\tt svds.m}
and 
{\tt irlba.m}  
are  based on the  
thick--restarted GKLB method 
\cite{baglama2005augmented}, 
there are subtle differences 
in the performance of 
{\tt svt\underline{\hspace{0.15cm}}svds.m} 
and 
{\tt svt\underline{\hspace{0.15cm}}irlba.m}
(see Section~\ref{sec:exp}).
Furthermore, 
since {\tt svt\underline{\hspace{0.15cm}}irlba.m}
provides its own standalone implementation of {\tt irlba}
it can also run in the freely available Octave, whereas
{\tt svt\underline{\hspace{0.15cm}}svds.m} cannot.
 This is particularly relevant 
since Octave's {\tt svds} 
is {\em not} a GKLB-based method 
and thus it cannot be used 
in our proposed framework. 
In fact, 
Octave's {\tt svds} 
creates the matrix {\tt C =[0\ A;A\tqs \ 0]} 
and calls Octave's {\tt eigs} function 
to compute the associated eigenpairs 
before translating them to the 
singular triplets of $A$. 
This implementation was 
also used by MATLAB's {\tt svds} 
at the time when 
\cite{cai2010-matrix-completion} 
first appeared $(2010)$
and was the rationale
for the authors 
to use PROPACK (GKLB-based method) 
in their SVT-MC Algorithm instead 
of that version of {\tt svds} 
citing a speedup of factor of ten 
\cite[Sec. 5.1.1]{cai2010-matrix-completion}.

We are only aware of 
one comparable MATLAB code 
that computes all singular triplets above a given user-inputted singular value threshold, 
{\tt svt.m} from \cite{li2017svt}. 
Similar to Octave's {\tt svds},  the code {\tt svt.m}  
uses {\tt C =[0\ A;A\tqs \ 0]} 
and calls MATLAB's {\tt eigs} function. 
As we will see 
in Section~\ref{sec:exp}, 
this is very inefficient and our methods consistently outperform {\tt svt.m}. 
With respect to {\tt svt\underline{\hspace{0.15cm}}irlba.R}, 
to the best of our knowledge 
this is the only 
{\sf R} implementation that computes all singular values 
above a specific threshold  
thus making it very valuable 
to diverse communities that are heavy {\sf R} users.

An added feature that further enhances 
the overall functionality of our codes 
is that instead of  providing a thresholding value 
that singular triplets must exceed
users may opt to specify a desired 
energy level of a PSVD, 
where $energy = \|\At\|_F^2/ \|A\|_F^2 = \Sigma_{i=1}^\pt \sigma_i^2/\|A\|_F^2$, with applications including imaging and  
nonlinear dimensionality reduction, 
see e.g. \cite{ji2016rank,sadek2012svd,feng2023fast} 
and the references therein. 
We note that under 
some mild conditions 
(Proposition~\ref{prop:nrmse-energy}) 
the notion of $energy$ 
is also closely related to another 
popular error measurement 
based on the Frobenius norm,
the normalized root mean squared error,
$nrmse = \|A - \At\|_F/ \|A\|_F$, \cite{erichson2019randomized,sadek2012svd}. 

\medskip
\begin{prop} \label{prop:nrmse-energy}
{\it Given $A_k$ with either $\AVerr$ or 
$\ATUerr$ equal to zero \eqref{eq:psvd}, then
$\|A - \At\|_F^2 = \|A\|_F^2 - \|\At\|_F^2$.}
\end{prop}

\begin{proof}
Without loss of generality, 
assume $\AVerr=0$ in \eqref{eq:psvd}. 
Recall the matrix--Pythogoras Lemma  \cite[Lemma 2.1]{boutsidis2014near} which states that if 
$X,Y \in \R^{m \times n}$ and $Y^TX =0$, then  
$\|X+Y\|_F^2 = \|X\|_F^2 + \|Y\|_F^2$. The desired conclusion 
follows by setting $X = A^T - \At^T$, $Y^T=\At$,  and using \eqref{eq:psvd} with $\AVerr=0$ to show that $Y^TX =0$.
\end{proof}
From Proposition~\ref{prop:nrmse-energy}, 
it now follows that 
if $\AVerr=0$ and/or 
$\ATUerr=0$ in \eqref{eq:psvd}, 
then
\begin{equation} \label{energynrmse}
(nrmse)^2 = \|A - \At\|_F^2/ \|A\|_F^2 = \|A\|_F^2/ \|A\|_F^2 - \|\At\|_F^2/ \|A\|_F^2
= 1 - energy \, .
\end{equation}

\begin{remark} \label{rem:energy}
It is worth highlighting that 
for low-rank approximations
obtained via randomization, 
such as in 
{\rm \cite{feng2018faster}}, 
{\rm Proposition~\ref{prop:nrmse-energy}} 
corresponds to {\rm \cite[Thm.~1]{feng2018faster}}. 
Furthermore, it is worthwhile emphasizing 
that  the user-inputted threshold condition (singular value or energy percentage, or normalized root mean
squared error) {\em is not} 
used as stopping criteria within {\tt svds.m or irlba.m} methods in the proposed algorithm in this paper, 
but rather as an exit criterion 
for the wrapper codes {\tt svt\underline{\hspace{0.15cm}}svds.m}, {\tt svt\underline{\hspace{0.15cm}}irlba.m}, and {\tt svt\underline{\hspace{0.15cm}}irlba.R}. That is, the $\pt$--PSVD of $A$ is  computed to within a prescribed {\tt svds.m or 
irlba.m} method's tolerance {\rm ('tol' in \rm Table \ref{table1})} 
and then tested against the user-inputted  threshold condition. 
This differs from other routines, 
e.g. {\tt farPCA}\footnote{Date: June 28, 2024: GitHib: \url{https://github.com/THU-numbda/farPCA}} 
{\rm \cite{feng2023fast}}, 
that use the energy percentage or normalized root mean
squared error as a stopping criteria, 
and ultimately makes any 
comparison with respect to timing and accuracy
between those methods and ours 
 difficult to interpret {\rm (see comparison 
 of one-sided errors in Example~\ref{example3})}. 
We note that 
in applications such as imaging 
where the use of 
\eqref{energynrmse} as an  stopping criteria 
to compute a PSVD of a matrix is appropriate, 
randomized SVD methods like 
{\tt farPCA} {\rm \cite{feng2023fast}} 
are  competitively fast 
and ought to be considered. 
\end{remark}

The paper is organized as follows. 
The next section contains our main contribution 
in the form of a hybrid algorithm 
(detailed codes are avilable at GitHub \url{https://github.com/jbaglama/svt}), 
while in Section~\ref{sec:par} we 
provide a detailed description of 
all parameters and 
include a sample of one-line computer code function calls. 
Finally, in Section~\ref{sec:exp} we present several numerical experiments with application to matrix completion and image compression, 
while Section~\ref{sec:conc} contains concluding remarks.

\section{Algorithms} \label{sec:alg}
In this section we only describe 
simplifications 
of our codes and refer the reader 
to publicly available documentation on GitHub
for further details. 
Our Algorithm~\ref{alg1} is a hybrid scheme that repeatedly calls a GKLB-based method and based on certain criteria calls a block SVD power ({\tt blksvdpwr}) method: Algorithm~\ref{alg2}. 
We first discuss the GKLB-based method, the criteria which connects the methods, and then finally the {\tt blksvdpwr} method.

At the center of our Algorithm~\ref{alg1} 
is the ability to 
repeatedly call a GKLB-based method 
(from now on generically referred to as a {\psvdns} method) 
to compute a PSVD of a matrix 
while taking the advantage of 
explicit deflation techniques 
\cite{baglama2023explicit}. 
Our codes 
{\tt svt\underline{\hspace{0.15cm}}svds.m}, 
{\tt svt\underline{\hspace{0.15cm}}irlba.m}. 
and 
{\tt svt\underline{\hspace{0.15cm}}irlba.R} 
are virtually the same modulo 
the choice of the  ${\psvdns}$ method 
and programming environment, 
i.e., ${\psvdns} \in \left\{ {\tt svds.m}, \, {\tt irlba.m}, \, {\tt irlba.R} \right\}$. 

We postpone providing 
detailed descriptions for 
all input {\em parameters} 
until Section~\ref{sec:par} 
(see Table~\ref{table1}) -- 
for now, we assume that 
they are all 
set to  their default values 
other than the possible input of an optional initial PSVD which inherently alters 
the start of Algorithm~\ref{alg1} (lines 3--7). 
The variable $\ell$ holds the current size of the PSVD of $A$ during an iteration, 
while the 
parameter structure {\em opts} 
(Algorithm~\ref{alg1}: line 17) possesses the tolerance, maximum iteration, subspace dimension, and 
a 
starting vector for 
a {\psvdns} method. 
Given these parameters a {\psvdns} method 
may not converge for all requested $\pt$ singular triplets. 
If some number of  
singular triplets have converged, then Algorithm~\ref{alg1} readjusts $\pt$ to that number  and continues. 
In the rare event  
that a {\psvdns} method  
returns with {\em no} 
converged singular triplets,  we increase the maximum number of iterations and the size of the subspace dimension before recalling a {\psvdns} method once again.
Through numerous experiments we have found that only one additional recall of a {\psvdns} method is required to get at least one singular triplet to converge. 
However, if after that one additional iteration a {\psvdns} method 
still produces no converged singular vectors, then Algorithm~\ref{alg1} exits at line 17  indicating {\psvdns} method failed (see Table \ref{table1}). 

\begin{algorithm}
\caption{\quad Hybrid Singular Value Threshold}
\label{alg1}
\begin{algorithmic}[1]
\State {\textbf{Input:}}
	$A$, threshold condition ({\it sigma}  or  {\it energy}); 
{\it parameters} (\,see Section \ref{sec:par}).
\smallskip
\State {\bf{Output:}} $U$; $S$; $V$.
\smallskip
\hrule 
\smallskip 
\If{nonempty$([U_0,S_0,V_0])$}
\State  $V=V_0$; $U=U_0$; $S=S_0$; $\ell = k_0$;
\If{$pwrsvd>0$}
\State $[U,S,V] = {\tt blksvdpwr}(A,V,U,pwrsvd)$;
\EndIf
\Else
\State
$A_d = A$; $V=[\ ]$; $U=[\ ]$; $S=[\ ]$; $\ell = 0$;
\EndIf
\While{$\ell \leq \min(m,n)$}
\If{$\ell > 0$ and $m \leq n$}
\State $A_d = A \!- U(U^TA)$
\ElsIf{$\ell > 0$ and $m > n$}
\State $A_d^T = A^T \!- V(AV)^T$
\EndIf
\State $[U_1,S_1,V_1]$ = {\psvdns}$(A_d, k$, opts);
\qquad \% \,${\psvdns} \, \in \, \{ {\tt svds.m, \, irlba.m, \, irlba.R}\}$
\State C1: $\ell >  0$ and ${\rm max}(|V_1^TV|\;\text{or}\;|U_1^TU|) > \sqrt{\epsilon}/(\ell+k)$;
\State C2: $\min(S_1) < \max(S)\sqrt{\epsilon}$;
\State C3: {\psvds} returned $>0$ and $< k$ values 
\If{\textup{(\,C1 or C2 or C3 or {\it pwrsvd} $>0$}\,)}
\State $[U,S,V] \!= \!{\tt blksvdpwr}(A,[V\ V_1],[U\ U_1],\max( 1, pwrsvd ))$;
\Else
\State $U =[U\ U_1]$; $V = [V\ V_1]$; $S={\tt blkdiag}(S,S_1)$;
\EndIf
\State $\ell = \ell + \rm{size}(S_1)$;
\If{$\min(S)\!<${\it sigma} or $\Sigma_{i=1}^k \sigma_i^2/\|A\|^2_F \geq energy$}
\State truncate and return $[U,S,V]$;
\Else
\State $k = k+incre$; $incre=2\cdot incre$;
\EndIf
\EndWhile
\end{algorithmic}
\end{algorithm}

The deflation schemes from \cite{baglama2023explicit} (Algorithm~\ref{alg1}: lines 13 and 15) are not explicitly computed, but rather set--up as internal matrix--vector product routines. 
For {\tt svds.m} or {\tt irlba.m} this is done via an internally written function handle, 
while for {\tt irlba.R} we provide two options, the {\tt irlba.R} parameter {\it mult} and a custom matrix multiplication function ({\it cmmf}) for the user to choose 
(see Table \ref{table1}). We refer to the {\tt irlba.R} documentation for details on {\it mult} and {\it cmmf} \cite{lewis2021irlba}.

We now turn our attention to 
the {\em hybrid} aspect of our proposed 
method, 
namely the 
use of a block SVD power method, {\tt blksvdpwr} from  Algorithm~\ref{alg2}, 
within Algorithm~\ref{alg1}. 
Use of {\tt blksvdpwr} method 
is triggered either 
via the user-inputted parameter {\it pwrsvd} which forces   
its use on every iteration of 
Algorithm~\ref{alg1} 
or when a certain criteria 
is met 
(see three cases in Algorithm~\ref{alg1}, lines 18--20). 
Case C1 
computes the 
orthogonality ``level'' of the large basis vectors that are implicitly held orthogonal during 
the deflation process. 
From our experiments, we have noticed that the large basis vectors rarely lose 
orthogonality, though in some isolated cases this can happen, e.g.,  when $A$ is very ill--conditioned, 
a ``large'' tolerance is used, or if $A$ has a significant number of 
multiple interior singular values (see Example~\ref{example1}). 
The choice of required orthogonality ``level'' before 
{\tt blksvdpwr} gets called follows 
\cite[Thm. 5]{larsen1998lanczos}. 
A more challenging 
scenario is the case C2 largely 
due to the fact that the deflation scheme (Algorithm~\ref{alg1}: lines 13 and 15) maps the large singular values to ``zero'' which under certain numerical conditions can reappear, i.e.,  be recomputed by a {\psvdns} method (see Example~\ref{example1}). 
Case C1 generally catches the recomputing of already mapped singular values, though  numerous experiments showed  that this may still occur when  too many singular triplets are requested with respect to the rank of $A$ or when $A$ has a very large number of interior multiple singular values. 
Finally, case C3 is more of a precautionary case which typically indicates that the GKLB process struggled computing interior values and therefore, securing a more stringent orthogonality level among basis vectors helps with the overall convergence of a 
{\psvdns} method on the next iteration.

With respect to 
{\tt blksvdpwr} itself, 
we note that our Algorithm~\ref{alg2} is based on the  
routine described in \cite{bentbib2015block} 
with slight differences in its implementation when it comes to  lines 9 and 16 of Algorithm~\ref{alg2}. 
We compute the full SVD of the upper 
triangular matrix $R$, whereas the authors in \cite{bentbib2015block}  use $R$ as the approximate diagonal matrix 
$S$. 
For a large matrix $R$, this computation can appear to be expensive, however the computation of the full SVD 
of $R$ is not anymore expensive in the overall scheme of Algorithm~\ref{alg1} where a {\psvdns} method, which depends on a GKLB-based method, repetitively computes the full SVD of a projected matrix of a similar size \cite{baglama2005augmented}. 
It is  worth highlighting here that while 
it is theoretically possible 
to use 
Algorithm~\ref{alg2} 
to compute a PSVD of $A$ that is not our 
objective here and no comparisons 
are made with respect to this aspect. 
In our scenario, 
the  sole purpose 
of Algorithm~\ref{alg2}  
within Algorithm~\ref{alg1} 
is simply to ensure the overall orthogonality 
of the basis vectors, check the reliability of a PSVD of $A$, and restore the one--sided GKLB structure. 
Moreover, an arbitrary SVD algorithm 
can not be used in place of Algorithm~\ref{alg2} 
as accuracy of Algorithm~\ref{alg1} heavily relies 
on having the one-sided GKLB structure. 
Therefore, the  output from Algorithm~\ref{alg2} is of the form \eqref{eq:psvd} with 
$\AVerr=0$ and $\ATUerr\neq 0$ 
when $m \leq n$ 
or with $\ATUerr=0$ and $\AVerr\neq0$ in case 
when  $m > n$.

Finally, the determination of next $\pt$ value and $incre$ (Algorithm~\ref{alg1}: line 30) follows the discussion in \cite{li2017svt}, 
where $incre$ is doubled on each iteration -- the user can only preset the initial values of $\pt$ and $incre$ on input 
(see Table~\ref{table1}). 
It is worth noting here that a priori knowledge of an initial $\pt$ can enhance overall convergence. 
Some promising ideas 
on how to estimate an initial $\pt$ for a {\psvdns} method are based on using harmonic Ritz values \cite{baglama2013implicitly} or eigenvalue estimates of $A^TA$ \cite{di2016efficient}. 
These are beyond the scope of this paper and at this time are 
an ongoing area of research. 
The algorithm terminates (Algorithm~\ref{alg1}: line 27) if it 
successfully computed 
all singular triplets above the user-inputted 
threshold {\it sigma} or 
if $\Sigma_{i=1}^k \sigma_i^2/\|A\|^2_F \geq energy$ for the user-inputted $0< energy\leq 1$.

\medskip
\begin{remark} \label{rem:irlbaR}
When $m > n$ and for the {\sf R} code {\tt svt\underline{\hspace{0.15cm}}irlba.R} one iteration of {\rm Algorithm~\ref{alg2}} must run each iteration of {\rm Algorithm~\ref{alg1}} to safely still use the deflation theory in {\rm \cite{baglama2023explicit}} applied to the smaller dimension side.
This is because of how the {\sf R} package {\tt irlba.R} 
{\rm \cite{lewis2021irlba}} handles matrix--vector products with the transpose of $A$. 
\end{remark}

\begin{algorithm}
\caption{\quad Block SVD Power Method ({\tt blksvdpwr})}
\label{alg2}
\vspace*{-4mm}
\begin{multicols}{2}
\begin{algorithmic}[1]
\State{\bf{Input:}} 
	$A \in \bRmn$;\ $V$; $U$; $iter$
\smallskip
\State {\bf{Output:}} $U$; $S$; $V$.
\smallskip
\hrule 
\smallskip 
\If{$m \leq n$}
\State [$U,R$] = {\tt qr}($U$);
\For{$j=1:iter$}
\State [$V,R$] = {\tt qr}($A^TU$);
\State [$U,R$] = {\tt qr}($AV$);
\EndFor
\State [$u_r,S,v_r$] = {\tt svd}($R$);
\Else
\State [$V,R$] = {\tt qr}($V$);
\For{$j=1:iter$}
\State [$U,R$] = {\tt qr}($AV$);
\State [$V,R$] = {\tt qr}($A^TU$);
\EndFor
\State [$v_r,S,u_r$] = {\tt svd}($R$);
\EndIf
\State return $V=Vv_r$; $U=Uu_r$; $S$;
\end{algorithmic}
\small{  ({\tt qr} is the economy-size QR decomposition)\\
({\tt svd} is the full SVD method)}

\end{multicols}
\vspace*{-2mm}
\end{algorithm}

\section{Input, Parameters, and Output} \label{sec:par}
All  our codes 
share a common set of key  
input parameters,   
though some parameters are 
language specific, 
e.g., MATLAB/Octave vs {\sf R} (see Table~\ref{table1}).
Of particular note are the parameters 
{\it kmax} and {\it psvdmax} 
whose purpose is to prevent 
excess computational time 
by limiting the number of requested 
singular values to 
a {\psvdns} method ({\it kmax}) 
and the size  
of the output $[U,S,V]$ 
({\it psvdmax}). 
The user can maintain very strong orthogonality 
among basis vectors 
and reset the one--sided GKLB structure 
by setting 
{\it pwrsvd} $>0$, 
otherwise the error growth is monitored
(Algorithm~\ref{alg1}: lines 18--20) 
and reset as needed. 
The output and variable 
{\it FLAG} that alerts the user to method's convergence are also given 
in  Table~\ref{table1}. 

\begin{table}[htb!]
\centering
{\small
\caption{Parameters and output 
for {\tt svt\underline{\hspace{0.15cm}}svds.m}, {\tt svt\underline{\hspace{0.15cm}}irlba.m}, and {\tt svt\underline{\hspace{0.15cm}}irlba.R}.}
\label{table1}
\vspace*{-1mm}
\begingroup
\setlength{\tabcolsep}{5.1pt} 
\renewcommand{\arraystretch}{0.9} 
\setlength\extrarowheight{1.9pt}
\begin{tabular}{|@{}c|p{4.4in}|}
 \hline
 \mc{2}{|c|}{
 Parameters for {\tt svt\underline{\hspace{0.15cm}}svds.m}, {\tt svt\underline{\hspace{0.15cm}}irlba.m}, and {\tt svt\underline{\hspace{0.15cm}}irlba.R}
 }\\ \hline
\mystrut{3.7mm}{\it sigma}  &   threshold -- if missing returns top $k$ singular triplets\\
{\it energy} &  energy percentage ($\leq 1$) (\ref{energynrmse}) (cannot be combined with {\it sigma}) \\
{\it tol} &  tolerance used for convergence in the {\psvds} method 
(default: $\sqrt{eps}$)\\
{\it k} &  initial $k$ used in the {\psvds} method (default: 6)\\
{\it incre} &  initial increment added to $k$ (default: 5)\\
{\it kmax}     &  maximum value $k$ can reach  (default: $\min\{0.1\!\cdot\! \min(m,n),100\}$)\\             
{\it p0} &  initial vector in the {\psvds} method  (default: {\it p0 = randn})\\
\,\,{\it psvdmax} &  max. dim. of output 
(default: $\max(\min(100+{\rm size}(\text{\it S0}),\min(n,m)),k)$)\\  
{\it pwrsvd} &  performs 
$pwrsvd$ iterations of Algorithm~\ref{alg2} 
(default: $0$)\\
{\it display} &  if set to $1$, 
then display some diagnostic information (default: $0$)\\ \hline
\mc{2}{|c|}{
Parameters for {\tt svt\underline{\hspace{0.15cm}}svds.m} and {\tt svt\underline{\hspace{0.15cm}}irlba.m}
}\\ \hline
{\it m}  &  number of rows of $A$ -- required if A is a function handle \\
{\it n} &  number of columns of $A$ -- required if A is a function handle \\
{\it U0} & left singular vectors of a previous PSVD of $A$ (default: [\ ])\\
{\it V0} & right singular vectors of a previous  PSVD of $A$ (default: [\ ])\\
{\it S0} & diagonal matrix of singular values of a previous PSVD of $A$ (default: [\ ])\\ \hline
\mc{2}{|c|}{Parameters for {\tt svt\underline{\hspace{0.15cm}}irlba.R}}\\ \hline
{\it psvd0}  &  a list of a previous PSVD of $A$ (default: NULL)\\
{\it cmmf}   &  
if TRUE uses an internal 
custom matrix multiplication 
function to compute the 
matrix--vector product deflation 
technique 
(Algorithm~\ref{alg1}, lines 15, 17). 
If FALSE, then 
uses an internal 
{\tt mult} parameter from 
{\tt irlba.R} 
to perform the deflation technique  (default: FALSE)\\ \hline
\mc{2}{|c|}{
Output for {\tt svt\underline{\hspace{0.15cm}}svds.m} and {\tt svt\underline{\hspace{0.15cm}}irlba.m}
}\\ \hline
\,\,\,\mystrut{3.5mm}{$U,S,V$} & left singular vectors, diagonal matrix of singular values, right singular vectors\\
{\it FLAG} &  $0$ -- successful output -- either threshold or energy percentage satisfied\\
\   &   $1$ -- {\psvdns} fails to compute any singular triplets -- output last values of 
                            $U,S,V$\\
\    &   $2$ -- {\it psvdmax} is reached  -- output last values of 
                            $U,S,V$\\
\   &   $3$ -- no singular values  above threshold {\it sigma} -- output $U=[\ ],V=[\ ],S=[\ ]$\\ \hline
\mc{2}{|c|}{
Output for {\tt svt\underline{\hspace{0.15cm}}irlba.R}
}\\ \hline
\mc{2}{|p{5.0in}|}{
Returns a list {\psvdRns} with $U$ as {\tt \psvdRns\$u}, $V$ as
{\tt \psvdRns\$v}, $S$ 
as a vector {\tt \psvdRns\$d}, 
and {\it FLAG} as {\tt \psvdRns\$flag} with same values as above.}\\ \hline
\end{tabular}
\endgroup}
\end{table}

To aid the reader, and also quickly 
demonstrate simplicity and versatility of our codes, 
we now provide several 
sample command line calls 
for each of them. 
However, 
in the publicly available repository, 
we exemplify how to call each routine 
using different combinations of 
available parameters 
and include demo files that reproduce 
all numerical results from Section~\ref{sec:exp} 
and some additional ones. 
Since MATLAB
{\tt svt\underline{\hspace{0.15cm}}svds.m} and MATLAB/Octave {\tt svt\underline{\hspace{0.15cm}}irlba.m} codes have {\em identical} command line calls, 
we only include one  for {\tt svt\underline{\hspace{0.15cm}}svds.m}.
Matrices $U$, $S$, $V$, and variable $FLAG$ 
(Table~\ref{table1}) 
are returned directly from our 
MATLAB/Octave routines, 
whereas 
{\tt svt\underline{\hspace{0.15cm}}irlba.R}
returns a {\em list} 
{\psvdRs} which can then be used to access 
$U$ as {\tt \psvdRns\$u}, $V$ as
{\tt \psvdRns\$v}, diagonal matrix $S$ 
as a vector {\tt \psvdRns\$d}, 
and {\it FLAG} as {\tt \psvdRns\$flag}. 
For example, MATLAB and {\sf R} commands, 
denoted by {\tt M>>}\, 
and {\tt R>}, respectively, 
used to compute 
all singular triplets with singular values 
exceeding $1.1$ are given by:

{{\tt M>> }}{\tt [U,S,V,FLAG] = svt\underline{\hspace{0.15cm}}svds(A, \tqs sigma\tqs, 1.1);}

{{\tt R> }}{\tt \ 
{\psvdRns} <- svt\underline{\hspace{0.15cm}}irlba(A, sigma = 1.1)}\\[1mm]
Analogously, 
command lines for 
computing 
all singular triplets with singular values exceeding $1.1$ given initial PSVD  
are as follows:

{{\tt M>> }}{\tt [U,\! S,\! V,\! FLAG]\! = \!svt\underline{\hspace{0.15cm}}svds(A,\!\! \tqs sigma\tqs,\!\! 1.1,\!\! \tqs U0\tqs,\!\! U,\!\! \tqs V0\tqs,\!\! V,\!\!\! \tqs S0\tqs,\!\! S);}

{{\tt R> }}{\tt \ 
{\psvdRns} <- svt\underline{\hspace{0.15cm}}irlba(A,\! sigma=1.1, \! 
{{\psvdzns}\! = \!{\psvdRns}})}\\[1mm]
Finally, for sample calls where energy percentage is 
used for thresholding purposes see Example~\ref{example3}.

\section{Numerical Examples} \label{sec:exp}
All  examples were carried out 
on a MacBook Pro (M1 Max) 
with 32GB of memory and macOS Sonoma 14.4.1. 
Computations were performed using  
MATLAB R2024a 
(Examples~\ref{example1}-\ref{example2}) 
or Rstudio–2023.06.1+524 (v4.3.2) (Example~\ref{example3}),
both with machine epsilon $\epsilon=2.2 \!\cdot\! 10^{-16}$. 
The recorded total cpu times (in seconds), 
denoted by $\Tcputot$, 
were done either with 
MATLAB {\tt tic-toc} or {\sf R} {\tt proc.time} commands. 
For Example~\ref{example1} 
multiple runs were used before the final 
results were recorded, 
while for 
Examples \ref{example2} and \ref{example3} 
only a single run was performed. 
The authors noticed that when the main comparison
of methods was timing, 
very little differences 
were observed between the runs and 
the reported results are typical.

\medskip
\begin{example} \label{example1} 
Here, 
we compare performance 
of 
our two codes 
{\tt svt\underline{\hspace{0.15cm}}svds.m} 
and 
{\tt svt\underline{\hspace{0.15cm}}irlba.m}
with 
the only other currently available comparable MATLAB alternative, 
{\tt svt.m}\footnote{\label{footnote-svt-github}Codes available at: \url{https://github.com/Hua-Zhou/svt}, retrieved on March 30, 2024.} \cite{li2017svt}. 
For comparison we used 
12 sparse matrices 
listed in Table~\ref{table3} 
from the SuiteSparse Matrix Collection 
\cite{davis2011university}
which vary in size, structure, and range/magnitude of singular values. 
The first six matrices,  
along with the corresponding 
choice of threshold values {\it sigma}, 
are exactly the same as used in 
\cite{li2017svt} 
thus allowing for a fair comparison. 
The low-rank perturbation $LR^T$ of the first five matrices, 
with $L$ and $R$ being random matrices  with $10$ columns, 
is consistent with examples in \cite{li2017svt} 
-- same RandStream used  
for {\tt svt.m}\footnoteref{footnote-svt-github}.

Each matrix 
problem was run 10 times.  
In all three codes, 
the common parameters 
{\it sigma}, {\it tol}, {\it k},  and {\it incre}, 
were uniformly set, 
while all other 
parameters were set to default with the exception of 
$tol=10^{-8}$,  $kmax=100$, and $psvdmax=800$ 
(see Table~\ref{table1}). 
In Table~\ref{table3} we also 
recorded the average cpu times for $\Tcputot$ 
and maximum values for errors 
$\Aerrtot 
\!=\!
\sqrt{\|\AVerr\|^2+\|\ATUerr\|^2}$ cf. (\ref{eq:psvd})
and 
$\UVerr
\!=\!
\sqrt{\|V^TV-I\|^2+\|U^TU-I\|^2}$.  
\begin{table}[htb!]
\centering
{\small
\caption{%
Example~\ref{example1}: Errors $\Aerrtot$ and $\UVerr$ 
and the total cpu timing $\Tcputot$ in seconds, 
for each of the codes that 
successfully computed all singular values above {\it sigma}.  Value ``$\star$'' indicates the code didn't successfully compute all the singular values.
} 
\label{table3}
\begingroup
\setlength{\tabcolsep}{0.7pt} 
\renewcommand{\arraystretch}{0.9} 
\setlength\extrarowheight{1.3pt}
\begin{tabular}{|l|c|c|c|c|c|c|c|c|c|c|c|c|} \hline
matrix  &  $sigma$ & \# & 
\mc{3}{c|}{\tt svt\underline{\hspace{0.15cm}}svds.m} &
\mc{3}{c|}{\tt svt\underline{\hspace{0.15cm}}irlba.m} &
\mc{3}{c|}{\tt svt.m}\\ 
\hhline{~~~|---------}
(size) &\  &SVs &$\Tcputot$ &$\Aerrtot$ &$\UVerr$ &$\Tcputot$ 
& $\Aerrtot$ &$\UVerr$ &$\Tcputot$ &$\Aerrtot$ &$\UVerr$ \\  
\hline
%
\makecell[l]{1. bibd\underline{\hspace{0.15cm}}20\underline{\hspace{0.15cm}}10$+LR^T$\\
($190\! \times\! 184756$)}&
466 &
20 & 
2.1s &
$10^{-11}$ &
$ 10^{-14}$ &
2.3s &
$ 10^{-11}$ &
$ 10^{-14}$ &
6.6s &
$ 10^{-5}$ &
$ 10^{-8}$\\ \hline
%
\makecell[l]{2. bibd\underline{\hspace{0.15cm}}22\underline{\hspace{0.15cm}}8$+LR^T$\\
($231\! \times\! 319770$)}&
435.79 &
20 & 
4.7s &
$10^{-10}$ &
$10^{-14}$ &
4.8s &
$10^{-10}$ &
$10^{-14}$ &
31.2s &
$10^{-5}$ &
$10^{-9}$\\ \hline
%
\makecell[l]{3. bfwb398$+LR^T$\\
($398\! \times\! 398$)}&
$2.2\!\cdot\! 10^{-5}$ &
50 & 
0.06s &
$10^{-12}$ &
$10^{-15}$ &
0.08s &
$10^{-12}$ &
$10^{-15}$ &
0.54s &
$10^{-12}$ &
$10^{-10}$\\ \hline
%
\makecell[l]{4. mhd4800b$+LR^T$\\
($4800\! \times\! 4800$)}&
$0.122$ &
50 & 
2.5s &
$10^{-8}$ &
$10^{-11}$ &
2.2s &
$10^{-8}$ &
$10^{-11}$ &
15.5s &
$10^{-8}$ &
$10^{-9}$\\ \hline
%
\makecell[l]{5. cryg10000$+LR^T$\\
($10000\! \times\! 10000$)}&
$21757$ &
50 & 
16.2s &
$10^{-4}$ &
$10^{-14}$ &
10.7s &
$10^{-4}$ &
$10^{-14}$ &
21.8s &
$10^{-4}$ &
$10^{-9}$\\ \hline
%
\makecell[l]{6. stormG2\underline{\hspace{0.15cm}}1000\\
 ($528185\! \times\! 1377306$)}&
$632.46$ &
50 & 
25.9s &
$10^{-6}$ &
$10^{-14}$ &
72.1s &
$10^{-6}$ &
$10^{-14}$ &
183.6s &
$10^{-11}$ &
$10^{-14}$\\ \hline
%
\makecell[l]{7. maragal\underline{\hspace{0.15cm}}2\\ ($555\! \times\! 350$)}&
$10^{-10}$ &
171 & 
0.34s &
$10^{-14}$ &
$10^{-15}$ &
0.45s &
$10^{-14}$ &
$10^{-15}$ &
$\star$ &
$\star$ &
$\star$\\ \hline
%
\makecell[l]{8. illc1033\\
($1033\! \times\! 320$)}&
$0.9$ &
197 & 
0.76s &
$10^{-8}$ &
$10^{-13}$ &
1.1s &
$10^{-9}$ &
$10^{-10}$ &
$\star$ &
$\star$ &
$\star$\\ \hline
%
\makecell[l]{9. well1850\\
($1850\! \times\! 712$)}&
$0$ &
712 & 
5.2s &
$10^{-8}$ &
$10^{-9}$ &
6.3s &
$10^{-8}$ &
$10^{-10}$ &
$\star$ &
$\star$ &
$\star$\\ \hline
%
\makecell[l]{ 10. JP\\
($87616\! \times\! 67320$)}&
$1500$ &
118 & 
31.1s &
$10^{-5}$ &
$10^{-15}$ &
32.4s &
$10^{-7}$ &
$10^{-15}$ &
55.5s &
$10^{-5}$ &
$10^{-9}$\\ \hline
%
\makecell[l]{11. Rel8\\
($345688\! \times\! 12347$)}&
$12.5$ &
47 & 
16.5s &
$10^{-7}$ &
$10^{-14}$ &
29.0s &
$10^{-7}$ &
$10^{-14}$ &
22.2s &
$10^{-7}$ &
$10^{-9}$\\ \hline
%
\makecell[l]{12. Rucci1\\
($1977885\! \times\! 109900$)}&
$6.5$ &
33 & 
55.1s &
$10^{-7}$ &
$10^{-14}$ &
114.6s &
$10^{-8}$ &
$10^{-14}$ &
75.2s &
$10^{-8}$ &
$10^{-9}$\\ \hline
\end{tabular}
\endgroup}
\end{table}

From Table~\ref{table3} 
it is evident that 
with respect to cpu timings 
{\tt svt\underline{\hspace{0.15cm}}svds.m} 
clearly outperforms 
the other two codes 
for all test matrices, 
especially when compared to 
{\tt svt.m} which is significantly slower. 
On the other hand,  
only a minor variation in timings 
between our two codes 
is not surprising, given their main difference is 
the choice of a {\tt psvd} method 
in Algorithm~\ref{alg1}. 
These two 
{\tt psvd} methods 
have the same underlying structure,  
and so 
the slight difference in 
timings in Table~\ref{table3} 
is accounted by differences in 
their  implementation, 
e.g., reorthogonalization of basis vectors, 
heuristics used for convergence, 
detection of multiple singular values, 
and selection of vectors to use for restarting 
\cite{baglama2005augmented,larsen1998lanczos}.

Turning our attention to 
errors 
$\Aerrtot$ 
and 
$\UVerr$, 
Table~\ref{table3} shows that 
both of our codes maintained strong orthogonality 
among both basis vectors, $U$ and $V$, 
which was not always the case for {\tt svt.m}.
Furthermore, 
{\tt svt.m} failed to 
compute 
all of the singular values above 
the prescribed threshold for 
three of the test matrices. 
In the case of the matrices 
{\em maragal\underline{\hspace{0.15cm}}2} 
and 
{\em well1850,} the number of 
needed singular values was very close to their ranks   
thus resulting in {\tt svt.m} 
recomputing mapped values -- 
the hybrid structure and 
use of {\tt blksvdpwr} routine prevented 
our codes succumbing to the same fate, emphasizing the need for such a check. 
Finally, for the matrix {\em illc1033}, 
which has  large number of multiple singular values clustered around $1.0$, 
{\tt svt.m} struggled and was only able to compute 
$128$ out of the $197$ 
singular values above $sigma = 0.9$. 
Setting  $tol=10^{-16}$ for {\tt svt.m} 
also did not result in success. 
\end{example}

\medskip
\begin{example} \label{example2} 
In this example 
we continue to compare our two MATLAB thresholding 
codes 
with {\tt svt.m} \cite{li2017svt}, 
but this time through the lens 
of an application to 
{\em matrix completion} 
briefly discussed in Section \ref{sec:intro}. 
We emphasize here that our objective 
is {\em not} 
to devise a new  competitive 
algorithm for matrix completion, but 
instead to illustrate 
 the ease at which 
our codes can be incorporated 
into larger numerical schemes 
and have a  meaningful impact. 
More specifically, 
we replace the main computational kernel, 
the PSVD section 
given by lines 5--8 in  \eqref{matrix-completion-alg}, 
in the original SVT-MC implementation 
\cite[{\rm Alg. 1}]{cai2010-matrix-completion}
with 
{\tt svt\underline{\hspace{0.15cm}}svds.m}, 
{\tt svt\underline{\hspace{0.15cm}}irlba.m}, 
and {\tt svt.m}. 
We do not compare 
here against alternatives 
that either modify other parts of SVT-MC Algorithm 
or ones that simply replace line 6 in \eqref{matrix-completion-alg}. For example, 
in \cite{feng2018faster} 
it appears that the authors utilize 
a randomized SVD method 
solely for line 6 in 
\eqref{matrix-completion-alg}
(publicly available code {\tt fastsvt}\footnote{Date: June 28, 2024. GitHub: \url{https://github.com/XuFengthucs/fSVT}\label{fnlabel-fastsvt}}) 
and do illustrate 
an impressive 
performance increase over 
when only {\tt svds} is used. 
But, as far as we can tell from 
the publicly available 
codes\footref{fnlabel-fastsvt}, 
{\tt fastsvt} constantly  recomputes the PSVD 
until the desired threshold is met, 
i.e., dismisses the knowledge of the 
previously computed PSVDs,  
thus making the comparison 
with our routines which are based  
on a completely different premise 
 beyond  the scope of this paper.

While several versions   
of the original SVT-MC Algorithm are publicly available (MATLAB, C, and FORTRAN),\footnote{
SVT--MC codes:    \url{https://www.convexoptimization.com/wikimization/index.php/Matrix_Completion.m}\\ 
\phantom{~~~~}SVT--MC codes:  \url{https://github.com/stephenbeckr/SVT/blob/master/SVT.m}} 
we only consider its MATLAB implementation. 
As 
stated in Section \ref{sec:intro}, at the time of publication 
\cite{cai2010-matrix-completion}, 
the authors used 
PROPACK \cite{larsen1998lanczos} 
as a method of choice for 
the PSVD computation (line 6 in
\eqref{matrix-completion-alg}) 
instead of the built--in MATLAB routine {\tt svds}. 
Since then, 
the {\tt svds} has changed 
and is now based on PROPACK with thick--restarting \cite{MATLAB:R2024a}, 
and so, for the purposes of comparison 
with our implementations of 
{\tt svt\underline{\hspace{0.15cm}}svds}, 
{\tt svt\underline{\hspace{0.15cm}}irlba}, 
and {\tt svt}, 
we use MATLAB's {\tt svds} 
for line 6 in \eqref{matrix-completion-alg}.

\begin{equation}\label{matrix-completion-alg}
\begin{array}{l} 
\ \ \ \ \ \ \ \ \ \ \vdots  \\
\left.\begin{array}{l}
5. \ \ \text{\textbf{repeat}} \\
6. \ \ \  \ \text{Compute\ } [U^{k-1},S^{k-1},V^{k-1}]_{s_k}\\
7. \ \ \  \ \text{Set\ } s_k = s_k + \ell   \\
8. \ \ \text{\textbf{until}\ } \sigma_{sk-\ell}^{k-1} \leq \tau \\
\end{array}\right\}\\
\ \ \ \ \ \ \ \ \ \ \vdots  \\[2mm]
\end{array}
\left.\begin{array}{l}
\text{PSVD  section \,}\\
\text{of \textbf{SVT-MC}} \\
\text{Algorithm} \\ \text{\cite[\text{Alg. 1}]{cai2010-matrix-completion}}
\end{array}\right\}
\begin{array}{l}
\text{replace\ with:\ }\\ \text{\tt svt\underline{\hspace{0.15cm}}svds}\\  \text{\tt svt\underline{\hspace{0.15cm}}irlba}\ \text{or}\\  \text{\tt svt}
\end{array}
\end{equation}
\vspace*{-3mm}

In order to analyze, 
within a larger matrix completion framework, 
the performance of our proposed 
alternatives for the PSVD section 
in (\ref{matrix-completion-alg}), 
we closely follow
the recommendations from 
\cite{cai2010-matrix-completion} 
when it comes to selecting test matrices 
and all of the parameters. 
We do so since 
an exhaustive presentation 
with all of the choices 
for different parameters is simply infeasible, 
e.g., 
$\tau$ (line 8  
in \ref{matrix-completion-alg}) 
was chosen heuristically in \cite{cai2010-matrix-completion}, 
though different values for it  
could lead to more/less iterations of the SVT-MC Algorithm or even no convergence, see \cite[Table 5.2]{cai2010-matrix-completion}. 
 We chose our 
test matrix $M$ to have the same 
structure (distributed Gaussian entries) 
as in \cite{cai2010-matrix-completion} and made it rectangular instead of square, 
i.e., $M = M_LM_R$,  where $M_L \in \R^{2000 \times r}$, $M_R \in \R^{r \times 20000}$, and $r$ is a chosen rank (ranging from $1\%$  to $15\%$) with 
an oversampling ratio of $4$ (see \cite[Table 5.1]{cai2010-matrix-completion}).

All four runs of 
SVT-MC Algorithm had 
successfully recovered the rank 
-- we used $10^{-3}$ 
for an overall convergence tolerance and for all PSVD methods tolerance was set to 
$10^{-8}$ 
and initial  $k=s_k$. 
For each of the implementations, in 
Table~\ref{table4} 
we recorded 
two cpu timings -- 
$\Tcputot$ for the overall convergence and $\Tcpupsvd$ for just the PSVD computations (lines 5--8 in \eqref{matrix-completion-alg}). 
Easy analysis shows that in virtually all of the cases roughly 90\%-95\% of overall cpu timing is 
due to its respective PSVD section, 
thus clearly indicating the importance of developing flexible portable routines that can 
be easily integrated 
as computational kernels 
in other algorithms, such as the ones presented here. For all choices of parameters {\em rank} and {\em sampling} in Table~\ref{table4}  our code {\tt svt\underline{\hspace{0.15cm}}irlba.m} significantly outperformed others, especially as rank increased.

\begin{table}[htb!]
\centering
{\small
\caption{Example~\ref{example2}: Comparison of the PSVD codes  
used in SVT--MC Algorithm for 
matrix $M = M_LM_R$, where $M_L \in \R^{2000 \times r}$ and $M_R \in \R^{r \times 20000}$.  
$\Tcputot$ is overall timing and $\Tcpupsvd$ timing just for the PSVD method (lines 5--8 in 
\eqref{matrix-completion-alg}).} 
\label{table4}
\vspace*{-3mm}
\begingroup
\setlength{\tabcolsep}{5pt} 
\renewcommand{\arraystretch}{.9} 
\setlength\extrarowheight{2pt}
\begin{tabular}{|l|c|c|c|c|c|c|c|c|} \hline
rank (\%)  &   \mc{2}{c|}{\tt svds.m}  &
\mc{2}{c|}{\tt svt\underline{\hspace{0.15cm}}svds.m} &
\mc{2}{c|}{\tt svt\underline{\hspace{0.15cm}}irlba.m} &
\mc{2}{c|}{\tt svt.m}\\ 
\hhline{~|--------}
 sample (\%)    & 
$\Tcputot$ & 
$\Tcpupsvd$  & 
$\Tcputot$  & $\Tcpupsvd$  &
$\Tcputot$  & $\Tcpupsvd$   &
$\Tcputot$  & $\Tcpupsvd$   \\ 
\hline
%
\makecell[l]{20 \ ($1  \%$)\\
$4  \%$} &
1453.1s &
1305.7s&
1509.3s &
1360.3s&
902.1s &
753.1s&
1740.1s &
1593.6s
\\ \hline
\makecell[l]{50 \ ($2.5  \%$)\\
$11  \%$} &
147.6s &
133.03s&
143.3s &
128.2s&
132.5s &
118.0s&
222.6s &
208.1s
\\ \hline
\makecell[l]{100 ($5  \%$)\\
$22  \%$} &
371.1s &
345.3s&
333.0s &
307.1s&
311.1s &
285.5s&
647.5s&
621.8s
\\ \hline
\makecell[l]{200 ($10  \%$)\\
$44  \%$} &
1341.4s &
1296.9s&
1020.4s &
976.4s&
981.3s &
937.4s&
1842.5s &
1799.4s
\\ \hline
\makecell[l]{300 ($15  \%$)\\
$65  \%$} &
2933.5s &
2877.2s&
1930.5s &
1874.5s&
1867.9s &
1812.1s&
3276.5s &
3220.7s
\\ \hline

\end{tabular}
\endgroup}
\end{table}
\end{example}

\medskip
\begin{example} \label{example3} 
In our final example, 
we investigate performance of the {\sf R}-implementation  
of our  Algorithm~\ref{alg1}, 
{\tt svt\underline{\hspace{0.15cm}}irlba.R}, 
and demonstrate its versatility 
of using energy percentage \eqref{energynrmse} 
as a thresholding condition. 
For that purpose, we consider 
the classical 
problem of image compression 
and analyze {\tt svt\underline{\hspace{0.15cm}}irlba.R}'s  performance  
against a comparable and highly 
popular {\sf R} package {\tt rsvd} 
\cite{erichson2019randomized}. 

For the sake of fair comparison, 
we consider the 
$1600 \times 1200$ 
grayscale image {\tt tiger}\footnote{Image is also freely available:
\url{https://en.wikipedia.org/wiki/File:Siberischer_tiger_de_edit02.jpg}. } 
that was used in \cite{erichson2019randomized} and
readily available 
from the {\tt rsvd} 
package:

{\tt R> }{\tt data("tiger", package = "rsvd")}\\[1mm] 
When initializing our 
{\tt svt\underline{\hspace{0.15cm}}irlba.R} code 
all parameters were kept  at 
default with the exception of 
$tol = 10^{-5}$ 
and $psvdmax = 1200$. 
Based on the discussion in \cite{erichson2019randomized}, 
we consider $nrmse = 0.12081$ 
as a thresholding consideration, 
which together with \eqref{energynrmse} 
in turn corresponds to 
energy percentage $98.54\%$
and is 
called 
via the following {\sf R} command:

{\tt R> }{\tt {\psvdRns} <- svt\underline{\hspace{0.15cm}}irlba(tiger, \!\!tol \!\!= \!\!1e-5,\!\! energy\!\! = \!\!0.9854, \!\!psvdmax \!= \!1200)}

\noindent
The compressed image can be reconstructed 
with the output from {\tt svt\underline{\hspace{0.15cm}}irlba.R} 
and displayed with the {\sf R} command {\tt image}, 
see \cite{erichson2019randomized} for details. 
In this case {\tt svt\underline{\hspace{0.15cm}}irlba.R} 
successfully returned an output with $k=100$ singular triplets (same $k$ value as in \cite{erichson2019randomized}) and
the desired $nrmse = 0.12081$ while requiring 
$\Tcputot = 6.2$s with errors 
$\Aerrtot=10^{-13}$ and $\UVerr = 10^{-14}$. 
We compared our results against  
{\tt rsvd} using its default parameter 
values though with one major distinction -- 
{\tt rsvd} requires a priori knowledge 
of the number of desired singular triplets, 
i.e., it required setting $k=100$. 
In this case, {\tt rsvd} clocked total cpu timing of  $\Tcputot = 1.2$s being several times 
faster than our method {\tt svt\underline{\hspace{0.15cm}}irlba.R}. 
While the {\tt rsvd} outputted a PSVD with 
$\UVerr = 10^{-14}$, it only achieved 
$nrmse = 0.12238$.  
This is the same output displayed within three significant digits as in 
\cite[Table~1]{erichson2019randomized} for {\tt rsvd} $q=2$ (default value). 
It is important to highlight that 
$\Aerrtot$ is {\em not} an appropriate 
error measure for {\tt rsvd} 
since $\|\AVerr\| = 10^{-13}$ 
and 
$\|\ATUerr\|=10^{0}$. 
Furthermore, {\tt rsvd} does not have
an input parameter for tolerance or error measurement -- this is checked outside the function. 
This observation 
also 
tends to plague  other numerical 
schemes that use energy as a stopping criteria, whereas in our case 
the energy level {\em is not} 
used 
to regulate the
underlying 
{\tt irlba.R}, {\tt irlba.m} and {\tt svds.m},  
but rather as an exit thresholding condition 
for the provided wrapper codes (see Remark~\ref{rem:energy}).

As a possible strategy to achieve 
a desired $nrmse$ value, 
as suggested in \cite[Table~1]{erichson2019randomized}, 
one may try to  increase 
the number of  
additional power iterations, $q$. 
On our first try, we set $q=3$  
and obtained an error of $nrmse = 0.12145$ 
which when rounded 
to three significant digits 
matches the error of $nrmse=0.121$ 
reported in 
\cite[Table~1]{erichson2019randomized}. 
However, if we wish to 
achieve the same accuracy 
in this error measurement ($nrmse$)
that meant setting $q=15$ 
which resulted in $\Tcputot = 6.3$s and $\UVerr = 10^{-14}$ with $nrmse = 0.12082$.
In other words, with respect to cpu timings 
{\tt svt\underline{\hspace{0.15cm}}irlba.R} 
performed at least as good as {\tt rsvd}, 
but with better errors and 
no required a priori knowledge on the number of needed singular values.

Now, one might take  the output from the 
first calculation 
and compute  the energy percentage 
$99\%$: 
 
{\tt R> }{\tt {\psvdRns} <- svt\underline{\hspace{0.15cm}}irlba(tiger,\!\! tol \!=\! 1e-5,\!\!  energy \!=\! 0.99,\!\!  psvdmax \!\!=\!\! 1200,\!\!  {\psvdzns} \!=\! {\psvdRns}})

\noindent
For {\tt svt\underline{\hspace{0.15cm}}irlba.R}  
this required additional $6.3$s, computing $155$ singular triplets with $nrmse = 0.09991$, $\Aerrtot=10^{-7}$ and $\UVerr = 10^{-14}$. 
Here the default value of $k=6$ is used 
for {\tt svt\underline{\hspace{0.15cm}}irlba.R} 
which may not always be the most efficient
and can be improved upon if there is 
a sensible 
guess for the number of 
additional initial number of 
singular triplets needed. 
Finally, we 
conclude that was not even 
an option with the {\tt rsvd} package.

\end{example}

\section{Concluding Remarks} \label{sec:conc} 

 We presented several publicly available 
well-documented software implementations 
(MATLAB and {\sf R}) 
of a new hybrid algorithm  demonstrating the ability to successfully and efficiently compute all of the singular values of $A$ above either a given threshold  or  a given energy percentage. 
Our codes consistently 
outperform 
comparable implementations 
tackling the same problem 
and in many cases 
the difference is quite substantial. 
Finally, the robustness 
and flexibility of the codes, as demonstrated with applications to matrix completion and image compression in Examples~\ref{example2}-\ref{example3}, 
highlight their ease of use 
either as a standalone program  
or as a part of a more complex numerical routine.  

\medskip

\bmhead{Code Availability}
The software is available as 
{\tt svt\underline{\hspace{0.15cm}}svds.m}, 
{\tt svt\underline{\hspace{0.15cm}}irlba.m}, 
and 
{\tt svt\underline{\hspace{0.15cm}}irlba.R} 
from the author's GitHub account 
(\href{https://github.com/jbaglama/svt}{https://github.com/jbaglama/svt}).

\bibliography{BagCP2024_ARXIV_Version_Submission}

\end{document}